\documentclass{article}
\usepackage{amsmath,amssymb,amsthm,verbatim}

\newcommand{\C}{\mathbb{C}}
\newcommand{\Z}{\mathbb{Z}}
\newcommand{\R}{\mathbb{R}}
\newcommand{\cro}{\textrm{cr}}
\newcommand{\floor}[1]{\left\lfloor #1 \right \rfloor}
\newcommand{\ceil}[1]{\left\lceil #1 \right \rceil}

\newtheorem{thm}{Theorem}
\newtheorem{prop}[thm]{Proposition}

\newtheorem{lem}[thm]{Lemma}

\newtheorem*{defn}{Definition}

\title{On the Crossing Number of Complete Graphs with an Uncrossed Hamiltonian Cycle}
\author{Daniel M. Kane}

\begin{document}

\maketitle

In \cite{guy}, Guy conjectured that the crossing number of the complete graph was given by:
$$
\cro(K_n) = Z(n) := \frac{1}{4}\floor{\frac{n}{2}}\floor{\frac{n-1}{2}}\floor{\frac{n-2}{2}}\floor{\frac{n-3}{2}}.
$$
Guy proved his conjecture for all $n\leq 10$. This was later extended by Pan and Richter in \cite{k11} to $n\leq 12$. For general $n$, the most that is known is that $0.8594 Z(n) \leq \cro(K_n) \leq Z(n)$. The lower bound was proved by de Klerk et al. in \cite{lowerbound}. The upper bound can be obtained in a number of ways, one of which we present here.

Draw the $K_n$ on the double cover of a disc, glued along the boundary, which is of course isomorphic to the sphere. Place the $n$ vertices around the boundary of the circle, labeled clockwise by elements of $\Z/n$. Let $S_1=\{1,2,\ldots, \floor{n/2}\},S_2=\{\floor{n/2}+1,\ldots,n\}$. For two vertices $x,y$ connect them by the line segment in the top disc if $x+y\in S_1$ and by the segment in the bottom disc in $x+y\in S_2$. For generic locations of the vertices, no three edges intersect at a point, and we claim that the number of crossings is given by $Z(n)$.

Note that a pair of edges $(u,v)$ and $(t,w)$ cross if and only if the corresponding cords of the circle cross and if $u+v$ and $t+w$ are either both in $S_1$ or both in $S_2$. Letting the four points $t,u,v,w$ when listed in clockwise order be given by the labels $x,x+a,x+a+b,x+a+b+c$ for integers $1\leq a,b,c\leq n$ with $a+b+c<n$. We note that of the three pairs of edges we could divide these four points into, only $\{(x,x+a+b),(x+a,x+a+b+c)\}$ has a chance of crossing. Note also that each pair of edges can be represented like this in exactly four ways by picking each possible vertex as $x$. We note that unless $a+c=n/2$ (in which case it is easy to verify that the two edges are drawn on opposite discs and thus don't cross) that exactly 2 of these representations have $a+c<n/2$. For such values of $a,b,c$, we wish to know the number of $x$ that give crossing edge pairs.

The above edges cross if and only if $2x+a+b$ and $(2x+a+b)+(a+c)$ either both lie in $S_1$ or both lie in $S_2$. Consider the average of the number of such $x$, and the number of such $x$ obtained when the roles of $a$ and $c$ are reversed. We claim that this is the number of $y\in\Z/n$ for which $y$ and $y+(a+c)$ either both lie in $S_1$ or both lie in $S_2$. For $n$ odd, this follows immediately from noting that $2x+a+b$ takes each value exactly once. Similarly, if $a+c$ is odd, each $y$ can be written in exactly two ways as either $2x+a+b$ or $2x+b+c$. Otherwise, it suffices to show that the number of solutions with odd $y$ equals the number of solutions with even $y$. If $n\equiv 2\pmod{4}$ this is true because $z\in S_1$ if and only if $z+n/2\in S_2$, thus the condition holds for $y$ if and only if it holds for $y+n/2$. If $n\equiv 0\pmod{4}$ and $a+c$ is even, this follows from noting that $z\in S_1$ if and only if $n/2+1-z\in S_1$, thus $y$ will work if and only if $n/2+1-(y+(a+c))$ does. It is easy to verify that the number of such $y$ is given by $\floor{n/2}-(a+c)+\ceil{n/2}-(a+c) = n-2(a+c)$. Thus the number of crossings is one half the sum over $1\leq a,b,c\leq n$ with $a+b+c<n$ and $a+c<n/2$ of $n-2(a+c)$. This equals:
\begin{align*}
\frac{1}{2}&\sum_{a+c\leq n/2} |\{1\leq b \leq n-1-(a+c)\}|(n-2(a+c))\\
& = \frac{1}{2}\sum_{a+c\leq n/2} (n-1-(a+c))(n-2(a+c))\\
& = \frac{1}{2}\sum_{s=2}^{\floor{n/2}}(s-1)(n-1-s)(n-2s)\\
& = Z(n).
\end{align*}

This construction of the upper bound has the interesting property that it contains a Hamiltonian cycle passing around the edge of the disc, none of whose edges have any crossings. For this special case of graph diagrams we establish the following lower bound on number of crossings:

We prove the following theorem:
\begin{thm}\label{mainThm}
If a $K_n$ is drawn in the plane in such a way that it has a hamiltonian cycle that does not cross any edges (including its own), then it has at least $n^4\left(\frac{1}{48} - \frac{1}{16 \pi^2} \right) + O(n^3)$ crossings.
\end{thm}
It should be noted that the bound from Theorem \ref{mainThm} is asymptotic to $(0.0145...)n^4$, whereas the best proven lower bound for general drawings is $(0.0134...)n^4$, and the upper bound is $(0.0156...)n^4$. It should also be noted that the class of diagrams that we consider is a generalization of those considered by \cite{cyl}.

The basic idea of the proof is as follows. Let the vertices of the graph be labeled $1,\ldots,n$ in order along the non-crossing hamiltonian cycle. Each edge of the graph not on this cycle must be on either the inside or the outside. Consider a pair of such edges: $(a,c)$ and $(b,d)$. If $a$ and $c$ separate $b$ from $d$ along the circle, and if these two edges are drawn on the same side of the cycle, then they must cross. This reduces our problem to finding the solution to a certain MAX-CUT problem. We make the following definitions:

\begin{defn}
For $a,b,c,d\in \Z/n$, we say that $(a,c)$ \emph{crosses} $(b,d)$ if, $a,b,c,d$ can be assigned representatives $a',b',c',d'$ so that either $a'<b'<c'<d'<a'+n$ or $a'>b'>c'>d'>a'-n$.
\end{defn}

Note that $(a,c)$ crosses $(b,d)$ if and only if $(c,a)$ crosses $(b,d)$. This is because if (without loss of generality) $a'>b'>c'>d'>a'-n$, then $c'>d'>a'-n>b'-n>c'-n$. Similarly, $(a,c)$ crosses $(b,d)$ if and only if $(b,d)$ crosses $(a,c)$.

\begin{defn}
For positive integer $n$, let $G_n$ be the graph whose vertices are unordered pairs of distinct elements of $\Z/n$, and whose edges connect pairs $\{a,c\}$ and $\{b,d\}$ when $(a,c)$ crosses $(b,d)$.
\end{defn}

\begin{lem}\label{maxcutreductionlem}
If a $K_n$ is drawn in the plane in such a way that it has a hamiltonian cycle that does not cross any edges (including its own), then it has at least $|E(G_n)|-\textrm{MAX-CUT}(G_n)$ crossings.
\end{lem}
\begin{proof}
For every such drawing of a graph, label the vertices along the hamiltonian cycle by elements of $\Z/n$ in order. The edges of our $K_n$ now correspond to the vertices of $G_n$ in the obvious way. Let $S$ be the subset of the vertices of $G_n$ corresponding to edges of the $K_n$ that lie within the designated cycle. Note that any two vertices in $S$ or any two vertices not in $S$ connected by an edge, correspond to pairs of edges in the $K_n$ that must cross. Thus the number of crossings of our $K_n$ is at least
$$
|E(S,S)| + |E(\bar{S},\bar{S})| = |E(G_n)| - |E(S,\bar{S})| \geq |E(G_n)| - \textrm{MAX-CUT}(G_n).
$$
\end{proof}

We have thus reduced our problem to bounding the size of the solution of a certain family of MAX-CUT problems. We do this essentially by solving the Goemans-Williamson relaxation of a limiting version of this family of problems. To set things up, we need a few more definitions.

\begin{defn}
By $S^1$ here we will mean $\R/\Z$. Given $a,b,c,d\in S^1$ we say that $(a,c)$ crosses $(b,d)$ if $a,b,c$ and $d$ have representatives $a',b',c',d'\in \R$ respectively, so that either $a'>b'>c'>d'>a'-1$ or $a'<b'<c'<d'<a'+1$. Define the indicator function
$$
C(a,b,c,d):= \begin{cases}1 & \textrm{if }(a,c)\textrm{ crosses }(c,d),\\ 0 & \textrm{otherwise} \end{cases}
$$
\end{defn}

We now present the continuous version of our MAX-CUT problem:

\begin{prop}\label{contProp}
Let $f:S^1\times S^1 \rightarrow \{\pm 1\}$, then
$$
\int_{(S^1)^4} f(w,y)f(x,z)C(w,x,y,z)dwdxdydz \geq \frac{-1}{\pi^2}.
$$
\end{prop}

We prove this by instead proving the following stronger result:

\begin{prop}\label{relaxedProp}
Let $f:S^1\times S^1 \rightarrow \C$ satisfy $|f(x,y)|\leq 1$ for all $x,y$, then
$$
\int_{(S^1)^4} f(w,y)\overline{f(x,z)}C(w,x,y,z)dwdxdydz \geq \frac{-1}{\pi^2}.
$$
Furthermore, for any $L^2$ function $f:S^1\times S^1\rightarrow\C$, we have that
\begin{align}\label{mainEqn}
\int_{(S^1)^4}  f(w,y)\overline{f(x,z)}C(w,x,y,z)&dwdxdydz \notag \\ & \geq \frac{-2}{\pi^2} \int_{(S^1)^2} |f(x,y)|^2\sin^2(\pi(x-y))dxdy.
\end{align}
\end{prop}

The proof of Proposition \ref{relaxedProp} will involve looking at the Fourier transforms of the functions involved. Before we can begin with this we need the following definition:
\begin{defn}
Define the function
$$
e(x):= e^{2\pi i x}.
$$
\end{defn}

We now express the Fourier transform of $C$.
\begin{lem}\label{FTLem}
We have that $C(w,x,y,z)$ is equal to:
\begin{align*}
& \frac{-1}{2\pi^2}\sum_{n,m\in \Z\backslash \{0\}} \frac{1}{nm}\left( e(nw -nx +my -mz) + e(nw -mx +my -nz)\right) \\
+ & \frac{1}{2\pi^2}\sum_{n,m\in \Z\backslash \{0\}} \frac{1}{nm}\left(e(-mx + (n+m)y -nz) +e(nw + my -(n+m)z)\right)\\
+ &\frac{1}{2\pi^2}\sum_{n,m\in \Z\backslash \{0\}}\frac{1}{nm}\left( e((n+m)w -nx - mz) - e(nw -(n+m)x +my)\right)\\
+ & \frac{1}{3}.
\end{align*}
\end{lem}
\begin{proof}
For $x\in\R/\Z$ let $[x]$ be the representative of $x$ lying in $[0,1]$. For any $w,x,y,z\in \R/\Z$, it is clear that $[x-w]+[y-x]+[z-y]+[w-z]\in\Z.$ It is not hard to see that this number is odd if and only if $(w,y)$ crosses $(x,z)$. Thus,
\begin{align*}
(-1)^{C(w,x,y,z)} &= e\left(\frac{[x-w]}{2} + \frac{[y-x]}{2} + \frac{[z-y]}{2} + \frac{[w-z]}{2}\right) \\ &= e\left(\frac{[x-w]}{2}\right)e\left(\frac{[y-x]}{2}\right)e\left(\frac{[z-y]}{2}\right)e\left(\frac{[w-z]}{2}\right).
\end{align*}
In order to compute the Fourier transform, we compute the Fourier transform of each individual term. Note that
\begin{align*}
\int e\left(\frac{[x-w]}{2}\right)e(-nx-mw)dxdw & = \int e\left(\frac{\alpha}{2}\right)e(-mw-n(\alpha+w))d\alpha dw\\
& = \int e(-(m+n)w-(n-1/2)\alpha)d\alpha dw\\
& = \frac{\delta_{m,-n}}{\pi i (n-1/2)}.
\end{align*}
Therefore, by standard Fourier analysis, we can say that
$$
e\left(\frac{[x-w]}{2}\right) = \frac{i}{\pi} \sum_{a\in\Z} \frac{e(aw-ax)}{a+1/2}.
$$
We have similar formulae for $e\left(\frac{[y-x]}{2}\right),e\left(\frac{[z-y]}{2}\right),$ and $e\left(\frac{[w-z]}{2}\right).$ Multiplying them together, we find that
$$
(-1)^{C(w,x,y,z)} = \frac{1}{\pi^4}\sum_{a,b,c,d\in\Z} \frac{e\left((a-d)w+(b-a)x+(c-b)y+(d-c)z \right)}{(a+1/2)(b+1/2)(c+1/2)(d+1/2)}.
$$

We now need to collect like terms. In particular, for every $4$-tuple of integers $\alpha,\beta,\gamma,\delta$, the coefficient of $e(\alpha w+\beta x+\gamma y + \delta z)$ equals the sum over $4$-tuples of integers $a,b,c,d$ with $\alpha=a-d$,$\beta=b-a$,$\gamma=c-b$,$\delta=d-c$ of
$$
\frac{1}{\pi^4(a+1/2)(b+1/2)(c+1/2)(d+1/2)}.
$$
Clearly, there are no such $a,b,c,d$ unless $\alpha+\beta+\gamma+\delta=0$. If this holds, then all such $4$-tuples are of the form $n,n+\beta,n+\beta+\gamma,n+\beta+\gamma+\delta$ for $n$ an arbitrary integer. Thus, we need to evaluate
$$
\frac{1}{\pi^4} \sum_{n\in\Z} \frac{1}{(n+1/2)(n+\beta+1/2)(n+\beta+\gamma+1/2)(n+\beta+\gamma+\delta+1/2)}.
$$

Consider the complex analytic function
$$
g(z) = \frac{ \pi\cot(\pi z)}{(z+1/2)(z+\beta+1/2)(z+\beta+\gamma+1/2)(z+\beta+\gamma+\delta+1/2)}.
$$
Note that along the contour $\max(|\Re(z)|,|\Im(z)|)=m+1/2$ for $m$ a large integer, $|g(z)|=O(m^{-4})$. Thus the limit over $m$ of the integral of $g$ over this contour is 0. This implies that the sum of all residues of $g$ is 0. Note that $g$ has poles only when either $z$ in an integer or when $(z+1/2)(z+\beta+1/2)(z+\beta+\gamma+1/2)(z+\beta+\gamma+\delta+1/2)=0$. At $z=n$, $g$ has residue
$$
 \frac{1}{(n+1/2)(n+\beta+1/2)(n+\beta+\gamma+1/2)(n+\beta+\gamma+\delta+1/2)}.
$$
Thus,
$$
\sum_{n\in\Z} \frac{1}{(n+1/2)(n+\beta+1/2)(n+\beta+\gamma+1/2)(n+\beta+\gamma+\delta+1/2)} = -\sum_{\rho\not\in\Z} \textrm{Res}_\rho(g).
$$

Therefore, $(-1)^{C(w,x,y,z)}$ equals
\begin{align*}
\frac{-1}{\pi^4}\sum_{\alpha+\beta+\gamma+\delta=0}e(\alpha w+\beta x+\gamma y+\delta z)\sum_{\rho\not\in\Z}\textrm{Res}_\rho(f_{\alpha,\beta,\gamma,\delta}).
\end{align*}

Note that all other such residues are at half integers. Note furthermore that $\cot(\pi z)$ is an odd function around half integers. Thus, $g$ has a residue at $z\not\in \Z$ only if $z$ is a root of $(z+1/2)(z+\beta+1/2)(z+\beta+\gamma+1/2)(z+\beta+\gamma+\delta+1/2)$ of even order, and in particular order at least 2. In other words, we have residues only when some pair of elements of $(0,\beta,\beta+\gamma,\beta+\gamma+\delta)$ are the same, but no three of them are unless all four are 0. In particular, we get residues in the following cases:

\begin{itemize}
\item When $\beta=0$, let $\alpha=n,\gamma=m$. Then, for $(\alpha,\beta,\gamma,\delta)=(n,0,m,-(n+m))$, we have a residue at $\rho=-1/2$ of $\frac{\pi^2}{nm}$ so long as $n,m\neq 0$.
\item When $\gamma=0$, let $\beta=-n,\delta=-m$. Then, for $(\alpha,\beta,\gamma,\delta)=(n+m,-n,0,-m)$, we have a residue at $\rho=n-1/2$ of $\frac{\pi^2}{nm}$ so long as $n,m\neq 0$.
\item When $\delta=0$, let $\alpha=n,\gamma=m$. Then, for $(\alpha,\beta,\gamma,\delta)=(n,-(n+m),m,0)$, we have a residue at $\rho=n-1/2$ of $\frac{\pi^2}{nm}$ so long as $n,m\neq 0$.
\item When $\alpha=0$, let $\beta=-n,\delta=-m$. Then, for $(\alpha,\beta,\gamma,\delta)=(0,-n,n+m,-m)$, we have a residue at $\rho=-1/2$ of $\frac{\pi^2}{nm}$ so long as $n,m\neq 0$.
\item When $\alpha+\beta=0$, let $\alpha=n,\gamma=m$. Then for $(\alpha,\beta,\gamma,\delta)=(n,-n,m,-m)$, we have a residue at $\rho=n-1/2$ of $\frac{-\pi^2}{nm}$ so long as $n,m\neq 0$.
\item When $\beta+\gamma=0$, let $\alpha=n,\gamma=m$. Then for $(\alpha,\beta,\gamma,\delta)=(n,-m,m,-n)$, we have a residue at $\rho=-1/2$ of $\frac{-\pi^2}{nm}$ so long as $n,m\neq 0$.
\item When $\alpha=\beta=\gamma=\delta$, we have a residue at $\rho=-1/2$.
\end{itemize}
Thus we have that $(-1)^{C(w,x,y,z)}$ equals
\begin{align*}
& \frac{1}{\pi^2}\sum_{n,m\in \Z\backslash \{0\}} \frac{1}{nm}\left( e(nw -nx +my -mz) + e(nw -mx +my -nz)\right) \\
- & \frac{1}{\pi^2}\sum_{n,m\in \Z\backslash \{0\}} \frac{1}{nm}\left(e(-mx + (n+m)y -nz) +e(nw + my -(n+m)z)\right)\\
- &\frac{1}{\pi^2}\sum_{n,m\in \Z\backslash \{0\}}\frac{1}{nm}\left( e((n+m)w -nx - mz) - e(nw -(n+m)x +my)\right)\\
+ & D.
\end{align*}
For some constant $D$. Noting that $C(w,x,y,z)=\frac{1-(-1)^{C(w,x,y,z)}}{2}$, we have that $C(w,x,y,z)$ equals
\begin{align*}
& \frac{-1}{2\pi^2}\sum_{n,m\in \Z\backslash \{0\}} \frac{1}{nm}\left( e(nw -nx +my -mz) + e(nw -mx +my -nz)\right) \\
+ & \frac{1}{2\pi^2}\sum_{n,m\in \Z\backslash \{0\}} \frac{1}{nm}\left(e(-mx + (n+m)y -nz) +e(nw + my -(n+m)z)\right)\\
+ &\frac{1}{2\pi^2}\sum_{n,m\in \Z\backslash \{0\}}\frac{1}{nm}\left( e((n+m)w -nx - mz) - e(nw -(n+m)x +my)\right)\\
 + & D'.
\end{align*}
On the other hand, $D'=\int_{(S^1)^4}C(w,x,y,z).$ Note that given any $w,x,y,z$ distinct that of the three ways to partition $\{w,x,y,z\}$ into two pairs, exactly one gives a set of crossing pairs. Thus $C(w,x,y,z)+C(w,y,x,z)+C(w,x,z,y)$ equals $1$ except on a set of measure 0. Thus, since the integral of each of these is $D'$, we have that $3D'=1$, or that $D'=1/3.$ This completes the proof.
\end{proof}

\begin{proof}[Proof of Proposition \ref{relaxedProp}]
Since $f$ is $L^2$ we may write
$$
f(x,y) = \sum_{n,m\in\Z} a_{n,m}e(nx+my)
$$
for complex numbers $a_{n,m}$ with $\sum_{n,m} |a_{n,m}|^2<\infty$. Notice that replacing $f(x,y)$ by $\frac{f(x,y)+f(y,x)}{2}$ does not effect the left hand side of Equation \eqref{mainEqn}, and can only increase the right hand side. Thus we can assume that $f(x,y)=f(y,x)$, and therefore that $a_{n,m}=a_{m,n}$.

By Lemma \ref{FTLem}, the left hand side of Equation \eqref{mainEqn} is
\begin{align*}
& \frac{-1}{2\pi^2}\sum_{n,m\in \Z\backslash \{0\}} \frac{a_{n,m}\overline{a_{n,m}}+a_{n,m}\overline{a_{m,n}}}{nm}
+ \frac{1}{2\pi^2}\sum_{n,m\in \Z\backslash \{0\}} \frac{a_{0,n+m}\overline{a_{m,n}}+a_{n,m}\overline{a_{0,n+m}}}{nm}\\
+ &\frac{1}{2\pi^2}\sum_{n,m\in \Z\backslash \{0\}}\frac{a_{n+m,0}\overline{a_{n,m}}+a_{n,m}\overline{a_{n+m,0}}}{nm}
 +  \frac{a_{0,0}\overline{a_{0,0}}}{3}.
\end{align*}
Using $a_{n,m}=a_{m,n}$, this simplifies to
\begin{align*}
& \frac{-1}{\pi^2}\sum_{n,m\in\Z\backslash \{0\}} \frac{a_{n,m}\overline{a_{n,m}}-a_{n,m}\overline{a_{n+m,0}}-a_{n+m,0}\overline{a_{n,m}}}{nm}+\frac{|a_{0,0}|^2}{3}\\
= & \frac{-1}{\pi^2}\sum_{n,m\in\Z\backslash \{0\}} \frac{|a_{n,m}-a_{n+m,0}|^2-|a_{n+m,0}|^2}{nm}+\frac{|a_{0,0}|^2}{3}\\
= & \frac{-1}{\pi^2}\sum_{n,m\in\Z\backslash \{0\}} \frac{|a_{n,m}-a_{n+m,0}|^2}{nm} +\frac{1}{\pi^2} \sum_{k\in\Z}|a_{k,0}|^2 \left(\sum_{n+m=k}\frac{1}{mn}\right)+\frac{|a_{0,0}|^2}{3}.
\end{align*}

We claim that for $k\neq 0$ that $\sum_{n+m=k}\frac{1}{nm}=0.$ This can be seen by noting that
$$
\frac{1}{n(k-n)} = \frac{-1}{k}\left(\frac{1}{n} - \frac{1}{n-k} \right),
$$
producing a telescoping sum.

If $k=0$,
$$\sum_{n+m=k}\frac{1}{nm} = \sum_{n\in\Z\backslash \{0\}} \frac{-1}{n^2} = -2\zeta(2) = \frac{-\pi^2}{3}.$$

Therefore, the left hand side of Equation \eqref{mainEqn} is
\begin{align*}
& \frac{-1}{\pi^2}\sum_{n,m\in\Z\backslash \{0\}} \frac{|a_{n,m}-a_{n+m,0}|^2}{nm} -\frac{1}{\pi^2} |a_{0,0}|^2 \left(\frac{\pi^2}{3}\right)+\frac{|a_{0,0}|^2}{3}\\
& = \frac{-1}{\pi^2}\sum_{n,m\in\Z\backslash \{0\}} \frac{|a_{n,m}-a_{n+m,0}|^2}{nm}.
\end{align*}

The right hand side of Equation \eqref{mainEqn} is
\begin{align*}
& \frac{-1}{\pi^2}\sum_{n,m} \frac{a_{n,m}\left(\overline{2a_{n,m}}-\overline{a_{n+1,m-1}}-\overline{a_{n-1,m+1}} \right)}{2}\\
= & \frac{-1}{\pi^2} \sum_{n,m} {a_{n,m}\left(\overline{a_{n,m}}-\overline{a_{n+1,m-1}}\right)}\\
= & \frac{-1}{2\pi^2} \sum_{n,m} |a_{n,m}-a_{n+1,m-1}|^2\\
= & \frac{-1}{2\pi^2} \sum_{n,m} |(a_{n,m}-a_{n+m,0})-(a_{n+1,m-1}-a_{n+m,0})|^2.
\end{align*}

We now let $b_{n,m} = a_{n,m}-a_{n+m,0}$. Notice that $b_{0,k}=b_{k,0}=0$. Equation \eqref{mainEqn} is now equivalent to
$$
\sum_{n,m\neq 0} \frac{-|b_{n,m}|^2}{nm} + \sum_{n,m} \frac{|b_{n,m}-b_{n+1,m-1}|^2}{2} \geq 0.
$$
We will in fact prove the stronger statement that
$$
\sum_{nm>0}\frac{|b_{n,m}|^2}{nm} \leq \sum_{(n+1)m>0}\frac{|b_{n,m}-b_{n+1,m-1}|^2}{2}.
$$
We note by symmetry that we can assume that $n,m>0$. We also note that it suffices to prove for each $k>0$ that
\begin{equation}\label{onelevelEqn}
\sum_{n,m>0, n+m=k}\frac{|b_{n,m}|^2}{nm} \leq \sum_{n+1,m>0, n+m=k}\frac{|b_{n,m}-b_{n+1,m-1}|^2}{2}.
\end{equation}
For fixed $k$, let $c_n = b_{n,k-n}-b_{n-1,k-n+1}$. By the symmetry exhibited by the $b$'s, the right hand side of Equation \eqref{onelevelEqn} is
$$
\sum_{n=1}^{\lfloor k/2\rfloor} |c_n|^2.
$$
Meanwhile, the right hand side is
\begin{equation}\label{qformEqn}
\sum_{n=1}^{\lfloor k/2\rfloor} \frac{\left|\sum_{i=1}^n c_i \right|^2}{n(k-n)} +
\sum_{n=1}^{\lfloor (k-1)/2\rfloor} \frac{\left|\sum_{i=1}^n c_i \right|^2}{n(k-n)}.
\end{equation}
Thus, the right hand side is given by a quadratic form in the $c_1,\ldots,c_{\lfloor k/2 \rfloor}$ with positive coefficients. Therefore, the biggest ratio between the right and left and sides is obtained by the unique eigenvector of this quadratic form for which all $c_i$ are positive. We claim that this happens when $c_n=k+1-2n$. For these $c$'s, the derivative of the expression in Equation \eqref{qformEqn} with respect to $c_m$ is
$$
2\sum_{n=m}^{\lfloor k/2\rfloor} \frac{\left(\sum_{i=1}^n c_i \right)}{n(k-n)} +
2\sum_{n=m}^{\lfloor (k-1)/2\rfloor} \frac{\left(\sum_{i=1}^n c_i \right)}{n(k-n)}.
$$
It is easy to verify that for this choice of $c_i$ that
$$
\sum_{i=1}^n c_i = n(k-n).
$$
Thus, the above reduces to
\begin{align*}
2\sum_{n=m}^{\lfloor k/2\rfloor} 1 +
2\sum_{n=m}^{\lfloor (k-1)/2\rfloor} 1 & = 2(\lfloor k/2\rfloor - m + 1) + 2(\lfloor (k-1)/2\rfloor - m + 1) \\ & = 2(k-2m+1) = 2c_m.
\end{align*}

Thus, these $c_i$ give the unique positive eigenvector. Hence it suffices to check Equation \eqref{onelevelEqn} when $c_m = k-2m+1$, or equivalently when $b_{n,k-n} = n(k-n)$. In this case, the left hand side of Equation \eqref{onelevelEqn} is
\begin{align*}
\sum_{n=1}^{k-1} n(k-n) & = \sum_{n=1}^{k-1} (kn-n^2)\\
& = \frac{k^2(k-1)}{2} - \frac{(k-1)k(2k-1)}{6}\\
& = \frac{k(k-1)(k+1)}{6}\\
& = \frac{k^3-k}{6}.
\end{align*}
For this choice, the right hand side is
\begin{align*}
\sum_{n=1}^k \frac{(k+1-2n)^2}{2} & = \sum_{n=1}^k \frac{k^2+2k-4kn+1-4n+4n^2}{2}\\
& = \frac{k^3}{2} + k^2 - k^2(k+1) + \frac{k}{2} - k(k+1) + \frac{k(k+1)(2k+1)}{3}\\
& = \frac{3k^3 + 6k^2 - 6k^3-6k^2 + 3k -6k^2-6k + 4k^3+6k^2+2k}{6}\\
& = \frac{k^3-k}{6}.
\end{align*}
Thus, the largest possible ratio between the left and right hand sides of Equation \eqref{onelevelEqn} is 1. This completes our proof.
\end{proof}

We are now prepared to prove our main theorem.

\begin{proof}[Proof of Theorem \ref{mainThm}]
We will proceed by way of Lemma \ref{maxcutreductionlem}. We note that $|E(G_n)|=n^4/24+O(n^3)$. We have only to bound the size of the MAX-CUT of $G_n$. Consider any subset $S$ of the vertices of $G_n$ defining a cut. We wish to bound the number of edges that cross this cut. Define the function $f_S:S^1\times S^1\rightarrow\{\pm 1\}$ as follows:
$$
f_S(x,y) = \begin{cases} 1 & \textrm{if }(\lfloor nx \rfloor,\lfloor ny \rfloor)\in S \\ -1 & \textrm{otherwise}  \end{cases}
$$

Consider
\begin{equation}\label{bigIntEqn}
\int_{(S^1)^4}f_S(w,y)f_S(x,z)C(w,x,y,z)dwdxdydz.
\end{equation}
In order to evaluate this expression, we consider the integral over the region $$R_{a,b,c,d}=[a/n,(a+1)/n]\times[b/n,(b+1)/n]\times[c/n,(c+1)/n]\times[d/n,(d+1)/n]$$ for some $a,b,c,d\in\Z/n$. We note that over this region that $f_S(w,y)f_S(x,z)$ is constant. In particular, it is $1$ if $(a,c)$ and $(b,d)$ are either both in $S$ or both not in $S$, and $-1$ otherwise. It should also be noted that if $a,b,c,d$ are distinct then $C(w,x,y,z)$ is also constant on this region, and in particular is $1$ if $G_n$ contains an edge between $(a,c)$ and $(b,d)$. Thus the expression in Equation \eqref{bigIntEqn} is
\begin{align*}
& \sum_{a,b,c,d} \int_{R_{a,b,c,d}}f_S(w,y)f_S(x,z)C(w,x,y,z)dwdxdydz \\
= & \sum_{\substack{a,b,c,d\\ \textrm{non-distinct}}}\int_{R_{a,b,c,d}} O(1) + \sum_{\substack{a,b,c,d \\\{\{a,c\},\{b,d\}\}\in E(G_n)}}\frac{f_S(a/n,c/n)f_S(b/n,d/n)}{n^4}\\
= & 8n^{-4}(|\textrm{Edges not crossing the cut}| - |\textrm{Edges crossing the cut}|) + O(n^{-1}).
\end{align*}

On the other hand, by Proposition \ref{contProp}, this is at least $\frac{-1}{\pi^2}$. Thus
$$
|\textrm{Edges crossing the cut}|-|\textrm{Edges not crossing the cut}| \leq \frac{n^4}{8\pi^2} +O(n^3).
$$
Adding the number of edges of $G_n$ and dividing by $2$, we find that
$$
|\textrm{Edges crossing the cut}| \leq n^4\left(\frac{1}{16\pi^2}+\frac{1}{48} \right) +O(n^3).
$$
This provides an upper bound on the size of $\textrm{MAX-CUT}(G_n)$. Thus by Lemma \ref{maxcutreductionlem}, the crossing number of $K_n$ is at least
$$
n^4\left(\frac{1}{48} - \frac{1}{16\pi^2} \right) +O(n^3).
$$
This completes our proof.
\end{proof}

\section*{Acknowledgements}

I would like to thank John Mackey for suggesting this problem to me. This work was done while the author was supported by a NSF postdoctoral fellowship.

\end{document}